\documentclass[12pt,twoside]{amsart}
\usepackage{amssymb,amsmath,amsthm}
\usepackage{verbatim}
\usepackage{graphicx}
\usepackage{epsfig, enumerate}
\usepackage{color}
\usepackage{soul}
\usepackage{ dsfont }
\usepackage[all]{xy}
\DeclareMathAlphabet{\mathpzc}{OT1}{pzc}{m}{it}
\newcommand{\N}{{\ensuremath{\mathbb{N}}}}

\def\A{\mathcal A}
\def\B{\mathcal B}
\def\F{\mathcal F}

\def\K{\mathcal K}
\def\L{\mathcal L}
\def\N{\mathcal N}

\def\p{\mathcal P}
\def\W{\mathcal W}
\def\H{\mathcal H}
\def\overF{\overline{\mathcal F}}

\def\d{\displaystyle}

\def\ep{\varepsilon}
\def\m{\mathpzc{m}}

\def\Na{\mathds{N}}

\def\vP{{\overset{\vee}P}}

\def\co{{\rm co}}

\newtheorem{theorem}{Theorem}[section]
\newtheorem{lemma}[theorem]{Lemma}

\newtheorem{corollary}[theorem]{Corollary}
\newtheorem{proposition}[theorem]{Proposition}

\newtheorem{example}[theorem]{Example}

\begin{document}
\title[]{$\A$-compact mappings}

\author{Pablo Turco}

\thanks{ This project was supported in part by CONICET PIP 0624 and PICT 2011-1465} 

 \address{IMAS - CONICET - Pab I,
 Facultad de Cs. Exactas y Naturales, Universidad de Buenos
Aires, (1428) Buenos Aires, Argentina.}
\email{paturco@dm.uba.ar}

\begin{abstract} 
For a fixed Banach operator ideal $\A$, we use the notion of $\A$-compact sets of Carl and Stephani to study $\A$-compact polynomials and $\A$-compact holomorphic mappings. Namely, those mappings $g\colon X\rightarrow Y$ such that every $x \in X$ has a neighborhood $V_x$ such that $g(V_x)$ is relatively $\A$-compact. We show that the behavior of $\A$-compact polynomials is determined by its behavior in any neighborhood of any point. We transfer some known properties of $\A$-compact operators to $\A$-compact polynomials. In order to study $\A$-compact holomorphic functions, we appeal to the $\A$-compact radius of convergence which allows us to characterize the functions in this class. Under certain hypothesis on the ideal $\A$, we give examples showing that our characterization is sharp.
\end{abstract}


\maketitle

\section*{Introduction}

In the theory of Banach operator ideals, some classes are characterized by the nature of their image on some neighborhoods of the origin of a Banach space. For example the classes of continuous, compact and weakly compact linear operators. These Banach operator ideals are called surjective. Based on this, several authors had introduced and studied different classes of functions between Banach spaces (in particular polynomials and holomorphic mappings) somehow extending this property, see for instance \cite{AMR, AR3, AR2, AR1, AS, GonGu00, LaTur1, Ryan_Wc}.

One of the first articles which consider polynomials and holomorphic functions of this type is due Aron and Schottenloher \cite{AS}. Here, the authors introduce compact polynomials and holomorphic mappings as follows. For Banach spaces $X$ and $Y$ and $x \in X$, a holomorphic function (resp. polynomial) $f\colon X\rightarrow Y$ is compact at $x$ if there exist $\ep>0$ such that $f(x+\ep B_X)$ is a relatively compact set in $Y$. Also, $f$ is said to be compact if it is compact at $x$ for all $x \in X$. In \cite{AS} several characterizations of compact holomorphic mappings and polynomials analogous to those of compact linear mappings are given. For instance, a holomorphic mapping is compact if and only if it is compact at the origin. Motivated by this work, Ryan \cite{Ryan_Wc} carried out a similar study of weakly compact holomorphic mappings, obtaining similar result to those in \cite{AS}. In 2000, Gonz\'alez and Guti\'errez \cite{GonGu00}, using the theory of {\it generating system of sets} of Stephani \cite{Ste}, extend some of the results of \cite{AS} and \cite{Ryan_Wc} to a wide class of holomorphic mappings.

In 1984, Carl and Stephani \cite{CaSt} introduced the notion of relatively $\A$-compact sets as follows. Given a (Banach) operator ideal $\A$, a set $K$ of a Banach space $X$ is relatively $\A$-compact if there exist a Banach space $Z$, a compact set $L\subset Z$ and a linear operator $T\in \A(Z;X)$ such that $K\subset T(L)$.  Sinha and Karn \cite{SiKa} defined, for $1\leq p <\infty$, relatively $p$-compact sets which, by \cite{LaTur2}, coincides with relatively $\mathcal N^p$-compact sets. Here $\mathcal N^p$ stands for the ideal of right $p$-nuclear operators, see for instance \cite[p.140]{RYAN} for the definition of this ideal.

Aron, Maestre and Rueda begin with the study of $p$-compact polynomials and $p$-compact holomorphic mappings, whose definition is obtained by generalizing in a natural way that of compact polynomials and holomorphic mappings. In light of \cite{AMR, AR1, LaTur1}, the behavior of $p$-compact polynomials is, in some sense, analogous to the behavior of $p$-compact operators. However, in \cite{LaTur1}, Lassalle and the author show that this is not the case for $p$-compact holomorphic mappings. For instance, \cite[Example~3.8]{LaTur1} exhibits a holomorphic mapping which is $p$-compact at the origin which fails to be $p$-compact. 

The aim of this work is to extend the results obtained in \cite{LaTur1} for $p$-compact mappings to the $\A$-compact setting. 

The article is organized as follows. In the first section we introduce some notation and state some basic results on $\A$-compact sets. In Section 2 we study $\A$-compact homogeneous polynomials. We show that this class fits in the theory of locally $\K_\A$-bounded homogeneous polynomials studied in \cite{AR2}. Also, we show that $\A$-compact homogeneous polynomials are a composition ideal of polynomials (see definitions below). This allows us to transfer some properties of $\A$-compact linear operators to $\A$-compact homogeneous polynomials. In particular, we show that $\A$-compact $n$-homogeneous polynomials forms a coherent sequence in the sense of Carando, Dimant and Muro (see definitions below).

Section 3 is dedicated to the study of $\A$-compact holomorphic functions. We define an $\A$-compact radius of convergence and we show that a function is $\A$-compact at some point if and only if all the polynomials of its Taylor series expansion at that point are $\A$-compact and the $\A$-compact radius of convergence is positive. As a counterpart, Example~\ref{exam_fun_no_A_compacta} shows a holomorphic mapping whose polynomials of its Taylor series expansion at any point are $\A$-compact, but the function fails to be $\A$-compact at every point. Also, we show that if a holomorphic function $f\colon X\rightarrow Y$ is $\A$-compact at some $x_0$, then $f$ is $\A$-compact for all $x \in x_0+rB_X$, where $r$ is the $\A$-compact radius of convergence of $f$ at $x_0$. This result is sharp, since Example~\ref{exam_A-compact_at_0} exhibits a holomorphic mapping $f\colon \ell_1\rightarrow X$ which is $\A$-compact at the origin, its $\A$-compact radius of convergence is $1$, but it fails to be $\A$-compact at $e_1$, the first canonical vector of $\ell_1$.

For general definitions concerning Banach operator ideals, we refer to the reader to the book of Pietsch~\cite{Pie} and of Defant and Floret~\cite{DF}.  Also, the books of Diestel, Jarchow and Tonge~\cite{djt} and of Ryan~\cite{RYAN}. For a general background on polynomials and holomorphic functions, we refer to the book of Dineen~\cite{DIN} and that of Mujica~\cite{Mu}.

\section{Preliminaries}

Throughout this paper, $X$, $Y$ and $Z$ are complex Banach spaces. We denote by $B_X$ the open unit ball of $X$. By $X', X'',\ldots$ we denote de topological dual, bidual, $\ldots$ of $X$. For a subset $M\subset X$, $\co\{M\}$ denotes the absolutely convex hull of $M$. 

We denote by $\L^n(X;Y)$ the space of all continuous $n$-linear mappings from $X$ to $Y$. As usual, $\L^1(X;Y)=\L(X;Y)$ is the space of all continuous linear operators and we identify $\L^0(X;Y)=Y$. A function $P\colon X\rightarrow Y$ is said to be a continuous $n$-homogeneous polynomial if there exists $U\in \L^n(X;Y)$ such that $P(x)=U(x,\ldots,x)$ for all $x\in X$.  With $\p^n(X;Y)$ we denote the vector space of all continuous $n$-homogeneous polynomials. Also, when $Y=\mathbb C$ we write $\p^n(X)$ instead of $\p^n(X;\mathbb C)$. The space $\p^n(X;Y)$ is a Banach space if it is endowed with the norm
$$
\|P\|=\sup_{x \in B_X} \|P(x)\|.
$$
Given $P \in \p^n(X;Y)$, $\vP$ stands for the unique symmetric continuous $n$-linear map, $\vP \in \L^n(X;Y)$ such that $\vP(x,\ldots,x)=P(x)$ for all $x \in X$.  If $x_0 \in X$, $P\in \p^n(X;Y)$ and $j< n$, we denote by 
$$
P_{x_0^j}(x)=\vP({x_0}^j,x^{n-j})=\vP(x_0,\overset{j \ {\rm times}} \ldots,x_0, x,\overset{n-j \ {\rm times}} \ldots,x).
$$
Note that $P_{x_0^j} \in \p^{n-j}(X;Y)$ and that $(P_{x_0^{j_1}})_{x_0^{j_2}}=P_{x_0^{j_1+j_2}}$ if $j_1+j_2 < n$.

An ideal of homogeneous polynomials $\mathcal Q$ is a subclass of all continuous homogeneous polynomials between Banach spaces such that, for all $n \in \Na$, the components $\mathcal Q^n(X;Y)=\p^n(X;Y) \cap Q$ satisfy:
\begin{enumerate}[\upshape (a)]
\item $\mathcal Q^n(X;Y)$ is a linear subspace of $\p^n(X;Y)$ which contains the $n$-homogeneous polynomials of finite type.
\item If $T\in \L(Z;X)$, $P \in \p^n(X;Y)$ and $R \in \L(Y;W)$, then the composition $T\circ P\circ R \in \mathcal Q^n(X;Y)$.
\end{enumerate}

$\mathcal Q$ is a Banach polynomial ideal if over $\mathcal Q$ is defined a norm $\|\cdot\|_{\mathcal Q}$ such that
\begin{enumerate}[\upshape (a)]
\item $\mathcal Q^n(X;Y)$ endowed with the norm $\|\cdot\|_{\mathcal Q}$ is a Banach space for all Banach spaces $X$ and $Y$.
\item $\|P^n\colon \mathbb C\rightarrow \mathbb C \colon P^n(x)=x^n\|_{\mathcal Q}=1$ for all $n \in \Na$.
\item If $T\in \L(Z;X)$, $P \in \p^n(X;Y)$ and $R \in \L(Y;W)$, then $\|T\circ P\circ R\|_{\mathcal Q}\leq \|T\| \|P\|_{\mathcal Q} \|R\|^n$.
\end{enumerate}
The case $n=1$ covers the classical theory of Banach operators ideals.

We denote by $\overF, \K$ and $\W$ the (Banach) ideals of approximable, compact and weakly compact operators. In general, we use $\A$ for a general Banach operator ideal, which is endowed with a norm $\|\cdot\|_\A$.

For a Banach operator ideal $\A$, an operator $T\in \L(X;Y)$ belongs to the surjective hull of $\A$, $\A^{sur}$, if there exist a Banach space $Z$ and a linear operator $R\in \A(Z;Y)$ such that $T(B_X)\subset R(B_Z)$. The norm on $\A^{sur}$ is defined as
$$
\|T\|_{\A^{sur}}=\inf\{\|R\|_\A \colon T(B_X)\subset R(B_Z)\},
$$
and makes it a Banach operator ideal. Also, we say that $\A$ is surjective if $\A=\A^{sur}$ isometrically.

For a Banach operator ideal $\A$, a set $M\subset X$ is $\A$-bounded if there exist a Banach space $Z$ and a linear operator $T\in \A(Z;Y)$ such that $M\subset T(B_Z)$, see \cite{Ste}. In the case of $\A=\W$, the $\W$-bounded sets coincide with the weakly compact sets and if $\A=\K$, we obtain the compact sets. A particular case of $\A$-bounded sets are the (relatively) $\A$-compact sets of Carl and Stephani, which were introduced in \cite{CaSt}. For a fix a Banach operator ideal $\A$, a subset $K$ of $X$ is said to be relatively $\A$-compact if there exist a Banach space $Z$, an operator $T\in \A(Z;X)$ and a compact set $M\subset Z$ such that $K\subset T(M)$. A sequence $(x_n)_n \subset X$ is $\A$-null if there exists a Banach space $Z$, an operator $T \in \A(Z;X)$ and a null sequence $(z_n)_n \subset Z$ such that $x_n=Tz_n$ for all $n \in \Na$. Also, from~\cite[Proposition~1.4]{LaTur2} we have that a sequence $(x_n)_n \subset X$ is $\A$-null if and only if, $(x_n)_n$ is relatively $\A$-compact and norm convergent to zero. The {\it size} of a relatively $\A$-compact set is defined in~\cite{LaTur2} as follows. For a relatively $\A$-compact set $K\subset X$, 
$$
\m_\A(K,X)=\inf\{\|T\|_\A \colon K\subset T(M), \ T\in \A(Z;X) \ {\rm and} \ M\subset B_X\},
$$
where the infimum is taken considering all Banach spaces $Z$, all operators $T \in \A(Z;X)$ and all compact sets $M\subset B_Z$ for which the inclusion $K\subset T(M)$ holds. If a set $K\subset X$ is not $\A$-compact, we set $\m_\A(K,X)=\infty$. Note that, if $K_1, K_2\subset X$ are relatively $\A$-compact sets and $\lambda \in \mathbb C$, we have the inequality
$$
\m_\A(\lambda K_1+K_2,X)\leq |\lambda|\m_\A(K_1,X)+\m_\A(K_2,X).
$$
Also, if $K\subset X$ is a relatively $\A$-compact set, then $\co\{K\}$ is relatively $\A$-compact and
$$
\m_\A(K,X)=\m_\A(\co\{K\},X).
$$ 
Observe that the relatively $\K$-compact and the relatively $\overF$-compact sets coincides with the relatively compact sets. Also, if $\A=\mathcal N^p$ ($1\leq p <\infty$), the relatively $\mathcal N^p$-compact set are precisely the relatively $p$-compact sets of Sinha and Karn \cite[Remark~1.3]{LaTur2}.
A linear operator $T\in \L(X;Y)$ is said to be $\A$-compact if $T(B_X)$ is a relatively $\A$-compact set in $Y$. The space of all $\A$-compact operators from $X$ to $Y$ is denoted by $\K_\A$. This space becomes a Banach operator ideal we endow it with the norm (see \cite{LaTur2})
$$
\|T\|_{\K_\A}=\m_\A(T(B_X),Y).
$$
In particular, for $1\leq p <\infty$, $\K_{\N^p}=\K_p$, the Banach ideal of $p$-compact operators of Sinha and Karn \cite{SiKa}.

Finally, recall that a function $f\colon X\rightarrow Y$ is holomorphic if for every $x_0 \in X$, there exist $r>0$ and a (unique) sequence of polynomials $P_nf(x_0) \in \p^n(X;Y)$ such that
$$
f(x)=\sum_{n=0}^{\infty} P_nf(x_0)(x-x_0)
$$
uniformly for $x \in x_0+rB_X$. This sum  is called the Taylor series expansion of $f$ at $x_0$. The space of all holomorphic function from $X$ to $Y$ is denote by $\H(X;Y)$ and, when $Y=\mathbb C$, $\H(X)$.

\section{$\A$-compact polynomials}

Aron and Rueda \cite{AR2} introduced the class of locally $\A$-bounded polynomials by considering the theory of generating system of sets of Stephani. For a Banach operator ideal $\A$, an $n$-homogeneous polynomial $P \in \p^n(X;Y)$ is locally $\A$-bounded if $P(B_X)$ is $\A$-bounded. The vector space of all locally $\A$-bounded $n$-homogeneous polynomials from $X$ to $Y$ is denote by $\p_\A^n(X;Y)$ and it becomes a Banach space endowed with the norm
$$
\|P\|_\A =\inf \|T\|_\A,
$$
where the infimun is taken over all the Banach spaces $Z$ and operators $T$ such that the inclusion
$$
P(B_X)\subset T(B_Z)
$$
holds \cite[Theorem~3.4]{AR2}. It is clear that, for a fixed Banach operator ideal $\A$, $\K_\A$-bounded sets are relatively $\A$-compact sets. Now, combining \cite[Theorem~2.1]{CaSt} with \cite[Corollary~1.9]{LaTur2}, relatively $\A$-compact sets are $\K_\A$-bounded sets. Hence $P \in \p_{\K_\A}^n(X;Y)$ if and only if $P$ maps bounded sets into relatively $\A$-compact sets. Every $\K_\A$-bounded homogeneous polynomial will be called $\A$-compact homogeneous polynomial. Also, from \cite[Corollary~1.9]{LaTur2} we may conclude that if $P \in \p_{\K_\A}^n(X;Y)$, then

$$
\|P\|_{\K_\A}= \m_\A(P(B_X),Y).
$$

Since $\K_\A$ satisfies Condition $\Gamma$ in~\cite[p. 965]{AR2} and $\K_\A$ is surjective~\cite[p. 82]{CaSt}, by \cite[Corollary~4.6]{AR2}  $\p^n_{\K_\A}$ is a {\it composition ideal} of polynomial. This last concept was introduced in \cite{BPR}. In the case of $\A$-compact $n$-homogeneous polynomials, we improve \cite[Corollary~4.6]{AR2} using the linearization of polynomials. We denote by $\widehat{\bigotimes}^n_{\pi_s}X$ the completion of the symmetric $n$-tensor product endowed with the symmetric projective norm $\pi_s$, $\Lambda \in \p^n(X;\widehat{\bigotimes}^n_{\pi_s}X)$ is defined by $\Lambda(x)=\otimes^n x$  and, for a $n$-homogeneous polynomial $P$, $L_P \in \mathcal L( \widehat{\bigotimes}^n_{\pi_s}X;Y)$ denotes the unique linear operator such that $P=L_P \circ \Lambda$. The following extends \cite[Lemma~2.1]{LaTur1}.

\begin{proposition}\label{Prop_A-comp-poly-lineal}
Let $X$ and $Y$ be Banach spaces and $\A$ be a Banach opeartor ideal. Then, $P\in \p_{\K_\A}^n(X;Y)$ if and only if $L_P \in \K_\A(\widehat{\bigotimes}^n_{\pi_s}X;Y)$. Moreover, $\|P\|_{\K_\A}=\|L_P\|_{\K_\A}$.
\end{proposition}
\begin{proof}
Note that  $B_{ \widehat{\bigotimes}^n_{\pi_s}X}=\co\{\otimes^n x \colon \|x\|\leq 1\}$ \cite[p.10]{Flo97}. Hence, th result follows by considering the inclusions
$$
P(B_X)\subset L_P(B_{\widehat{\bigotimes}^n_{\pi_s}X})\subset \co\{P(B_X)\}.
$$  
\end{proof}

As an immediate consequence, we have the following result.

\begin{corollary} 
Let $X$ and $Y$ be Banach spaces and $\A$ be a Banach operator ideal. Then, $P \in \p^n_{\K_\A}(X;Y)$ if and only there exist a Banach space $Z$, an $\A$-compact operator $T\in \K_\A(Z;Y)$ and an $n$-homogeneous polynomial $Q \in \p^n(X;Z)$ such that $P=T\circ Q$. Moreover, $\|P\|_{\K_\A}=\inf\{\|T\|_{\K_\A} \|Q\|\}$, where the infimum is taken over all the possible factorization of $P$ as above.
\end{corollary}

In the particular case of $p$-compact polynomials, the above corollary was obtained in \cite[Proposition~2.1]{LaTur1}.
 
Proposition~\ref{Prop_A-comp-poly-lineal} allows us to transfer some properties of the $\A$-compact operators to $\A$-compact polynomials. Recall that for a Banach operator ideal $\A$, the dual ideal of $\A$, $\A^d$ is given by
$$
\A^d(X;Y)=\{T \in \L(X;Y) \colon T' \in \A(Y';X')\}.
$$
This ideal becomes a Banach operator ideal if it is endowed with the norm
$$
\|T\|_{\A^d}=\|T'\|_\A.
$$
From \cite[Corollary~2.4]{LaTur2}, we have the isometric equality $\K_\A=\K_\A^{dd}$. Given a polynomial $P \in \p^n(X;Y)$, its transpose $P'\in \L(Y;\p^n(X))$ is given by $P'(y')(x)=y'\circ P (x)$, see \cite{AS}. 

\begin{proposition} Let $X$ and $Y$ be Banach spaces, $\A$ a Banach operator ideal and $P \in \p^n(X;Y)$.  Then $P\in \p^n_{\K_\A}(X;Y)$ if and only if $P' \in \K_\A^d(Y';\p^n(X))$. Moreover, we have $\|P\|_{\K_\A}=\|P'\|_{\K_\A^d}$.
\end{proposition}
\begin{proof}
If $P\in \p^n_{\K_\A}(X;Y)$, by Proposition~\ref{Prop_A-comp-poly-lineal}, $L_P \in \K_\A(\widehat{\bigotimes}^n_{\pi_s}X;Y)$ and $\|L_P\|_{\K_\A}=\|P\|_{\K_\A}$. 
By \cite[Corollary~2.4]{LaTur2}, it follows that $L_P \in \K_\A^d(Y';\widehat{\bigotimes}^n_{\pi_s}X)$ with $\|L'_P\|_{\K_\A^d}=\|L_P\|_{\K_\A}$. Since $P'=\Lambda'\circ L'_P$, then $P'\in \K_\A^d(Y';\p^n(X))$ and $$
\|P'\|_{\K_\A^d}\leq \|L_P\|_{\K_\A}=\|P\|_{\K_\A}.
$$

For the converse, note that $\Lambda'\colon (\widehat{\bigotimes}^n_{\pi_s}X))'\rightarrow \p^n(X)$ is an isometric isomorphism. Now, by assumption, we have that $P'=\Lambda'\circ L'_P \in \K_\A^d(Y';\p^n(X))$.  Then $\Lambda'^{-1} \circ P' =L'_P \in \K_\A^d(Y';(\widehat{\bigotimes}^n_{\pi_s}X)')$, and therefore $L''_P \in \K_\A((\widehat{\bigotimes}^n_{\pi_s}X)'';Y'')$. Applying again \cite[Corollary~2.4]{LaTur2}, we get that $L_P \in \K_\A(\widehat{\bigotimes}^n_{\pi_s}X;Y)$ which, thanks to Proposition~\ref{Prop_A-comp-poly-lineal}, is equivalent to $P \in \p_{\K_\A}(X;Y)$. Furthermore, 
$$
\|P\|_{\K_\A}=\|L_P\|_{\K_\A}=\|L'_P\|_{\K_\A^d}=\|\Lambda'^{-1} \circ P'\|_{\K_\A^d}\leq \|P'\|_{\K_\A^d}.
$$
\end{proof}

Also, the isometric equality $\K_\A=\K_\A^{dd}$ allows us to show that the Aron-Berner extension of an $n$-homogeneous polynomial is a $\|\cdot\|_{\K_\A}$-isometric extension which preserves the ideal of $\A$-compact polynomials. This extend \cite[Proposition~2.4]{LaTur1}. Recall that for a polynomial $P \in \p^n(X;Y)$, the Aron-Berner extension of $P$, $\overline P$, is the extension of $P$ from $X$ to $X''$ obtained by weak-star density \cite{AB}.

\begin{proposition} Let $X$ and $Y$ be Banach spaces, $\A$ a Banach operator ideal and $P \in \p^n(X;Y)$. Then, $P \in \p^n_{\K_\A}(X;Y)$ if and only if $\overline P \in \p_{\K_\A}^n(X'';Y'')$. Moreover, we have $\|P\|_{\K_\A}=\|\overline P\|_{\K_\A}$.
\end{proposition}
\begin{proof}
Suppose that $\overline P$ is $\A$-compact. Since $J_Y(P(B_X))\subset \overline P(B_{X''})$, then $J_Y(P(B_X))$ is a relatively $\A$-compact set and $\m_\A(J_Y(P(B_X));Y'')\leq \|\overline{P}\|_{\K_\A}$. By 
\cite[Corollary~2.3]{LaTur2}, $P(B_X)$ is relatively $\A$-compact in $Y$ and $\m_\A(P(B_X),Y)=\m_\A(J_Y(P(B_X)),Y'')$, which implies that $P \in \p^n_{\K_\A}(X;Y)$ and $\|P\|_{\K_\A}\leq \|\overline P\|_{\K_\A}$.

Now, take $P \in \p^n_{\K_\A}(X;Y)$. By Proposition~\ref{Prop_A-comp-poly-lineal}, $L_P$ is $\A$-compact with $\|L_P\|_{\K_\A}=\|P\|_{\K_\A}$. Since $\overline P=L''_P \circ \overline \Lambda $ and the Aron-Berner extension preserves the norm~\cite[Theorem~3]{DG} we have $\|\overline \Lambda\|=1$ and
$$
\overline P(B_{X''})=L''_P\circ \overline \Lambda (B_{X''})\subset L''_P(B_{(\widehat{\bigotimes}^n_{\pi_s}X)}'').
$$
By \cite[Corollary~2.4]{LaTur2}, $L_P''$ is $\A$-compact, which implies that $\overline P$ is $\A$-compact and
$$
\|\overline P\|_{\K_\A}\leq \|L''_P\|_{\K_\A}=\|L_P\|_{\K_\A}=\|P\|_{\K_\A},
$$
which completes the proof.

\end{proof}

Now we show that $\A$-compact $n$-homogeneous polynomials are complete determined by their behavior at some point. To this end, we introduce the following notation. For Banach spaces $X$ and $Y$, $\A$ a Banach operator ideal and $x_0 \in X$, we say that an $n$-homogeneous polynomial $P\in \p^n(X;Y)$ is $\A$-compact at $x_0$ if there exists $\ep>0$ such that $P(x_0+\ep B_X)$ is relatively $\A$-compact in $Y$. 

First note that $P$ is $\A$-compact at $0$ if and only if $P$ is $\A$-compact. And we claim that this happens if and only if $P$ is $\A$-compact at every $x \in X$. Indeed, suppose that for $\ep>0$, $P(\ep B_X)$ is relatively $\A$-compact. Then $\ep^n P(B_X)$ is relatively $\A$-compact, and hence $P(B_X)$ is also relatively $\A$-compact. Now, take $x \in X$. Then, since $x+\ep B_X \subset (\|x\|+\ep) B_X$, we have
$$
P(x+\ep B_X) \subset P((\|x\|+\ep )B_X) = \textstyle{(\frac{\|x\|+\ep}{\ep})}^n P(\ep B_X),
$$
thus $P(x+\ep B_X)$ is relatively $\A$-compact.

Now, to show the converse, we need the following lemma.

\begin{lemma}\label{lemma_P_a_Acompact}
Let $X$ and $Y$ be Banach spaces, $\A$ a Banach operator ideal, $P\in \p^n(X;Y)$ and $x_0 \in X$. If $P$ is $\A$-compact at $x_0$ then, for any $j<n$,  $P_{(x_0)^j} \in \p_{\K_\A}^{n-j}(X;Y)$.
Moreover, if $P(x_0+\ep B_X)$ is relatively $\A$-compact for some $\ep>0$, then
$$
\|P_{(x_0)^j}\|_{\K_\A} \leq C_{\ep, j}(x_0) \m_\A(P(x_0+\ep B_X),Y)
$$
where $C_{\ep, j}(x_0)= \d\frac{(n-j)!}{n! \ \ep^{nj-\frac{j  (j+1)}{2}}} (\|x_0\|+\ep)^{n(j-1) - \frac{j (j-1)}{2}}$.
\end{lemma}
\begin{proof}
Suppose that $P(x_0+\ep B_X)$ is relatively $\A$-compact for some $\ep>0$. First we will show that $P_{x_0}$ is $\A$-compact. For any $x \in X$  and $\xi \in \mathbb C$ is a primary $n$-root of the unit we have
$$
P_{x_0}(x)=\vP(x_0,x^{n-1})=\frac{1}{n^2 \ (n-1)^{n-1}} \sum_{j=0}^{n-1} \xi^j P(x_0+(n-1)\xi^j x),
$$
see \cite[Corollary 1.8]{CaDiMu}. Then,
$$
P_{x_0}(\frac{\ep}{n-1} B_X)\subset \frac{1}{n^2 \ (n-1)^{n-1}} \sum_{j=0}^{n-1} \xi^j P(x_0+\ep B_X),
$$
which shows that $P_{x_0}$ is $\A$-compact at $0$ and therefore, $\A$-compact. Moreover, 

$$
\begin{array}{rl}
\|P_{x_0}\|_{\K_\A}&\d =(\frac{n-1}{\ep})^{n-1}\m_\A(P_{x_0}(\frac{\ep}{n-1} B_X),Y)\\
&\d \leq (\frac{n-1}{\ep})^{n-1} \frac{1}{n^2 \ (n-1)^{n-1}} \m_\A(\sum_{j=0}^{n-1} \xi^j P(x_0+\ep B_X),Y)\\
&\d \leq \frac{1}{\ep^{n-1} \ n^2} \sum_{j=0}^{n-1} \m_\A(P(x_0+\ep B_X),Y)\\
&\d = \frac{1}{\ep^{n-1} \ n} \m_\A(P(x_0+\ep B_X),Y).
\end{array}
$$

Now, since $P_{x_0}$ is $\A$-compact, it is $\A$-compact at $x_0$. Reasoning as above, we get that $(P_{x_0})_{x_0}=P_{(x_0)^2} \in \p_{\K_\A}^{n-2}(X;Y)$ and

$$
\begin{array}{rl}
\|P_{(x_0)^2}\|_{\K_\A}&\d \leq \frac{1}{\ep^{n-2} \ (n-1)} \m_\A(P_{x_0}(x_0+\ep B_X),Y)\\
&\d \leq \frac{(\|x_0\|+\ep)^{n-1}}{\ep^{n-2} \ (n-1)}  \|P_{x_0}\|_{\K_\A}\\
&\d \leq \frac{(\|x_0\|+\ep)^{n-1}}{\ep^{n-2} \ (n-1)} \frac{1}{\ep^{n-1} \ n}\m_\A(P(x_0+\ep B_X),Y)\\
&\d = \frac{(\|x_0\|+\ep)^{n-1}}{n (n-1) \ \ep^{2n-3}} \m_\A(P(x_0+\ep B_X),Y).
\end{array}
$$

An inductive argument gives the result.
\end{proof}

\begin{proposition}\label{Prop_A-comp_poly}
Let $X$ and $Y$ be Banach spaces, $\A$ a Banach operator ideal and $P\in \p^n(X;Y)$. The following are equivalent.
\begin{enumerate}[\upshape (i)]
\item $P \in \p_{\K_\A}^n(X;Y)$.
\item There exist $x_0 \in X$ such that $P$ is $\A$-compact at $x_0$.
\item $P$ is $\A$-compact at $0$.
\end{enumerate}
\end{proposition}
\begin{proof}
We only need to prove that (ii) implies (iii). Suppose that $P$ is $\A$-compact at $x_0$. For any $x \in X$ we have
$$
P(x)=P(x+x_0-x_0)=\sum_{j=0}^{n} \genfrac{(}{)}{0pt}{}{n}{j}(-1)^j\vP(x_0^j,(x+x_0)^j)=\sum_{j=0}^{n} \genfrac{(}{)}{0pt}{}{n}{j}(-1)^jP_{(x_0)^j}(x+x_0),
$$
and
$$
P(B_X)\subset \sum_{j=0}^{n}(-1)^j \genfrac{(}{)}{0pt}{}{n}{j}P_{(x_0)^j}(x_0+B_X).
$$
By the above lemma,  $P_{(x_0)^j}$ is $\A$-compact for every $j<n$, and the result follows.
\end{proof}

The above proposition generalizes \cite[Proposition~3.3]{AMR}, where the result was obtained for $p$-compact mappings.

Carando, Dimant and Muro \cite{CaDiMu} (see also \cite{BoBrJuPe}) introduce the concept of {\it coherent} sequences of $n$-homogeneous Banach ideals of polynomials associated with Banach ideal of operators as follows. Given a Banach operator ideal $\A$, the sequence $(\A^n)_n$ of $n$-homogeneous polynomials is coherent and associated with $\A$ if there exist positive constants $A$ and $B$ such that for every Banach spaces $X$ and $Y$ the following conditions hold:
\begin{enumerate}[\upshape (a)]
\item $\A^1=\A$
\item For each $P \in \A^{n+1}(X;Y)$ and $x_0 \in X$, $P_{x_0} \in \A^n(X;Y)$ and
$$
\|P_{x_0}\|_{\A^n} \leq A \|P\|_{\A^{n+1}} \|x_0\|.
$$
\item For each $P\in \A^n(X;Y)$ and $x' \in X'$, $x'P \in \A^{n+1}(X;Y)$ and
$$
\|x'P\|_{\A^{n+1}}\leq B\|x'\| \|P\|_{\A^n}.
$$
\end{enumerate}

The sequence $(\p^n_{\K_\A})_n$, fulfills condition (c) with $B=1$. Also, $\p^1_{\K_\A}(X;Y)=\K_\A(X;Y)$. Finally, if $P \in \p^{n+1}_{\K_\A}(X;Y)$ and $x_0 \in X$, by Lemma~\ref{lemma_P_a_Acompact}, we have $P_{x_0} \in \p^n_{\K_\A}(X;Y)$ and
$$
\|P_{x_0}\|_{\K_\A}=\frac{1}{\ep^{n-1} \ n} \m_\A(P(x_0+\ep B_X),Y)\leq \frac{(\|x_0\|+\ep)^n}{\ep^{n-1} \ n} \|P\|_{\K_\A},
$$
for every $\ep>0$. Hence, choosing $\ep=n \|x_0\|$ in the above inequality, we have
$$
\|P_{x_0}\|_{\K_\A}\leq (\frac{1+n}{n})^n \|x_0\| \|P\|_{\K_\A} \leq e \|x_0\|  \|P\|_{\K_\A}.
$$
Summarizing, we have proved the following.

\begin{proposition}\label{Prop_A-comp coherent}
Let $\A$ be a Banach operator ideal. The sequence $(\p^n_{\K_\A})_n$ is a coherent sequence associated with the Banach ideal $\K_\A$.
\end{proposition}

\section{$\A$-compact holomorphic functions}

In this section we deal with $\A$-compact holomorphic functions. For a fixed Banach operator ideal $\A$, a holomorphic function $f\colon X\rightarrow Y$ is said to be $\A$-compact at $x \in X$ if there exist $\ep>0$ such that $f(x_0+\ep B_X)$ is relatively $\A$-compact in $Y$. If $f$ is $\A$-compact at every $x \in X$, then it is called $\A$-compact. We denote by $\H_{\K_\A}(X;Y)$ the space of all  holomorphic mappings between $X$ and $Y$. In the case when $\A=\K$ we cover the compact holomorphic function defined in \cite{AS} and, when $\A=\N^p$, $1\leq p <\infty$, we cover the $p$-compact holomorphic function defined in \cite{AMR} and studied in \cite{LaTur1}. Recently, Aron and Rueda introduced the $\A$-bounded holomorphic function \cite{AR3}. Is not hard to see that $\A$-compact and $\K_\A$-bounded holomorphic functions coincides. 

Aron and Schottenloher established the following result concerning compact holomorphic functions \cite[Proposition~3.4]{AS}.

\begin{proposition}[Aron-Schottenloher]\label{Prop_AS} Let $X$ and $Y$ be Banach spaces and $f \in \H(X;Y)$. The following are equivalent:
\begin{enumerate}[\upshape (i)]
\item $f$ is compact.
\item For all $n \in \Na$ and all $x \in X$, $P_nf(x)$ is an $n$-homogeneous compact polynomial.
\item For all $n \in \Na$ $P_nf(0)$ is an $n$-homogeneous compact polynomial.
\item $f$ compact at $0$.
\end{enumerate}
\end{proposition}

In \cite[Theorem~3.2]{Ryan_Wc} Ryan proved that the above result remains valid for functions in the class of weakly compact holomorphic mappings. Later, Gonz\'alez and Guti\'errez extended this result to the class of $\A$-bounded holomorphic mappings, whenever $\A$ is a closed surjective Banach operator ideal \cite[Proposition~5]{GonGu00}. Recall that closed Banach ideals are those which, endowed with the usual norm, forms a closed subspace. Here we will show some coincidences but, which is more interesting, some differences when dealing with the class of $\A$-compact holomorphic functions. Indeed, we will show that, in general $\A$-compact holomorphic functions behave more like nuclear than compact maps. Our result extend that of \cite{LaTur1} given in the context of $p$-compactness. In order to proceed, we will consider the $\A$-compact radius of convergence of a function $f$ at $x_0 \in X$, which can be see as natural extension of the Cauchy-Hadamard formula considering the $\A$-compact norm of polynomials. If $\sum_{n=0}^\infty P_nf(x_0)$ is the Taylor series expansion of $f$ at $x_0$, we say that
$$
r_{\K_\A}(f,x_0)=1/\limsup \|P_nf(x_0)\|_{\K_\A}^{1/n}
$$
is the radius of $\A$-compact convergence of $f$ at $x_0$. If $f$ is not $\A$-compact at $x$, then $r_{\K_\A}(f;x)=0$. When $\A=\N^p$, we cover the definition of $p$-compact radius of convergence \cite[Definition~3.2]{LaTur1}.
  
The key to show that (i) implies (ii) or (iv) implies (iii) in Proposition~\ref{Prop_AS} relies in the following lemma, whose proof can be found in \cite[Proposition 3.4]{AS}.

\begin{lemma}\label{Lemma_AMR_Cauchy}
Let $X$ and $Y$ be Banach spaces and $f\in \H(X;Y)$. For every $x_0 \in X$, $\ep >0$ and $n \in \Na$ the following inclusion holds
$$
P_nf(x_0)(\ep B_X) \leq \overline{\co\{f(x_0+\ep B_X)\}}. 
$$
\end{lemma}

In particular, Lemma~\ref{Lemma_AMR_Cauchy} implies that if $P_nf(x_0)$ fails to be $\A$-compact for some $n \in \Na$, then $r_{\K_\A}(f,x_0)=0$. Also, if $f$ in $\H(X;Y)$ is $\A$-compact at $x_0$ then, for all $n \in \Na$, $P_nf(x_0)$ is $\A$-compact and, for each $n \in \Na$,
$$
\ep^n \|P_nf(x_0)\|_{\K_\A} =\m_\A(P_nf(x_0)(\ep B_X),Y)\leq \m_\A(\overline{\co\{f(x_0+\ep B_X)\}};Y).
$$
Therefore $r_{\K_\A}(f;x_0)>0$. It turns out that the converse also holds. To see this we need the following lemma, which extends \cite[Lemma~3.1]{LaTur1}.

\begin{lemma}\label{Lemma_Suma_A-compactos}
Let $X$ be a Banach space, $\A$ a Banach operator ideal and consider $(K_n)_n$ a sequence of relatively $\A$-compact sets in $X$. If $\sum_{j=1}^{\infty} \m_\A(K_j,X)<\infty$, then the set $K=\{\sum_{j=1}^{\infty} x_j \colon x_j \in K_j\}$ is relatively $\A$-compact and $\m_\A(K,X)\leq \sum_{j=1}^{\infty} \m_\A(K_j,X)$.
\end{lemma}
\begin{proof}
Fix $\ep>0$ and $j \in \Na$. There exist a Banach space $Z_j$, a compact set $M_j\subset B_{Z_j}$ and a linear operator $T_j \in \A(Z_j;X)$ such that $K_j\subset T_j(M_j)$ and $\|T_j\|_{\A}\leq (1+\ep)\m_\A(K_j,X)$. Since $\sum_{j=1}^{\infty} \m_\A(K_j,X)<\infty$, there exists a sequence $(\beta_j)_j \in B_{c_0}$ such that $\sum_{j=1}^{\infty} \frac{1}{\beta_j} \m_\A(K_j,X)\leq (1+\ep)\sum_{j=1}^\infty \m_\A(K_j,X)$. Define the Banach space $Z=\prod_{j=1}^{\infty}{Z_j}$ endowed with the norm $\|z\|_Z=\sup_{j \in \Na} \|z_j\|_{Z_j}$, if $z=(z_j)_j$. Define the linear operator $T\colon Z\rightarrow X$ as $T=\sum_{j=1}^{\infty} \frac{1}{\beta_j} T_j \circ \pi_j$, where $\pi_j\colon Z\rightarrow Z_j$ is the canonical projection, $\|\pi_j\|=1$. Since
$$
\begin{array}{rl}
\d \sum_{j=1}^{\infty} \|\frac{1}{\beta_j} T_j \circ \pi_j\|_\A & \d \leq \sum_{j=1}^{\infty} \|\frac{1}{\beta_j} T_j\|_\A\\
&\d \leq (1+\ep)\sum_{j=1}^{\infty} \frac{1}{\beta_j}\m_\A(K_j;X)\\
&\d \leq (1+\ep)^2 \sum_{j=1}^{\infty} \m_\A(K_j;X) <\infty,
\end{array}
$$
then $T\in \A(Z;X)$ and $\|T\|_\A \leq (1+\ep)^2\sum_{j=1}^{\infty} \m_\A(K_j;X)$. Finally, consider the compact set $M\subset B_Z$ defined by $M=\prod_{j=1}^{\infty} \beta_j M_j$. Thus, $K\subset T(M)$, which implies that $K$ is a relatively $\A$-compact set. Moreover,
$$
\m_\A(K,X)\leq \|T\|_\A \leq (1+\ep)^2 \sum_{j=1}^{\infty} \m_\A(K_j,X).
$$
Since $\ep>0$ is arbitrary, the result follows.
\end{proof}

\begin{proposition}\label{Prop_A-compact_equiv} Let $X$ and $Y$ be Banach spaces, $\A$ a Banach operator ideal $x_0 \in X$ and $f \in \H(X;Y)$. The following are equivalent.
\begin{enumerate}[\upshape (i)]
\item $f$ is $\A$-compact at $x_0$.
\item For all $n \in \Na$, $P_nf(x_0) \in \p_{\K_\A}^n(X;Y)$ and $r_{\K_\A}(f,x_0)>0$.
\end{enumerate}
\end{proposition}

\begin{proof}
That (i) implies (ii) follows from Lemma~\ref{Lemma_AMR_Cauchy}. To see that (ii) implies (i), take $\ep>0$ such that $\ep < r_{\K_\A}(f,x_0)$ and $f(x)=\sum_{n=1}^{\infty} P_nf(x_0)(x-x_0)$, with uniform convergence in $x_0+\ep B_X$,. Hence, we have 
$$
f(x_0+\ep B_X)\subset \{\sum_{n=1}^{\infty} x_n \colon x_n \in P_nf(x_0)(\ep B_X)\}.
$$
By Lemma~\ref{Lemma_Suma_A-compactos}, to show that $f(x_0+\ep B_X)$ is a relatively $\A$-compact set, it enough to see that $\sum_{n=1}^{\infty} \m_\A(P_nf(x_0)(\ep B_X);Y) <\infty$, which is true since
$$
\sum_{n=1}^{\infty} \m_\A(P_nf(x_0)(\ep B_X);Y)=\sum_{n=1}^{\infty} \ep^n \m_\A(P_nf(x_0)(B_X);Y)=\sum_{n=1}^{\infty} \ep^n \|P_nf(x_0\|_{\K_\A}
$$
and $\ep < r_{\K_\A}(f,x_0)=1/\limsup\|P_nf(x_0)\|_{\K_\A}^{1/n}$.
\end{proof}

Since $\K_\A$ satisfy condition $\Gamma$, by \cite[Theorem~2.7]{AR3} we can describe the $\A$-compact radius of convergence of a function $f$ in $\H(X;Y)$ at $x_0 \in X$ as
$$
r_{\K_\A}(f;x_0)=\sup\{ t>0 \colon f(x_0+tB_X) \ {\rm is \ \A-compact}\}.
$$
With this, $\A$-compact holomorphic mappings have a natural local behavior which can be establish in terms of the radius of $\A$-compact convergence.

\begin{proposition}\label{Prop_A-comp1pt_y_otro}
Let $X$ and $Y$ be Banach spaces, $\A$ a Banach operator ideal, $x_0 \in X$ and $f \in \H(X;Y)$. If $f$ is $\A$-compact at $x_0$, then $f$ is $\A$-compact for all $x \in x_0 + r_{\K_\A}(f,x_0)B_X$. 
\end{proposition}

\begin{corollary}\label{Coro_A-comp1pt_A-comp}
Let $X$ and $Y$ be Banach spaces, $\A$ a Banach operator ideal, $x_0 \in X$ and $f \in \H(X;Y)$. If $f$ is $\A$-compact at $x_0$ and $r_{\K_\A}(f,x_0)=\infty$, then $f$ is $\A$-compact.
\end{corollary} 

In virtue of Proposition~\ref{Prop_A-comp1pt_y_otro}, is natural to ask, for a given function $f \in \H(X;Y)$ which is $\A$-compact at $x_0 \in X$, if it is $\A$-compact beyond $x_0+r_{\K_\A}(f,x_0) B_X$. Thanks to the Josefson-Nissenzweig theorem we have, for any Banach spaces $X$ and $Y$ and any Banach operator ideal $\A$, an $\A$-compact holomorphic mapping $f \in \H_{\K_\A}(X;Y)$, whose $\A$-compact radius of convergence at the origin in finite. It is enough to consider a sequence $(x'_n)_n \subset X'$ with $\|x'_n\|=1$ for all $n \in \Na$ and $(x'_n)_n$ pointwise convergent to $0$. Then $f(x)=\sum_{n=1}^{\infty} x'_n(x)^n$ belongs to $\H(X)$ and, since it takes values in $\mathbb C$, is $\A$-compact for any Banach operator ideal $\A$. Also, since $\|x'_n\|^n=\|x'_n\|^n_{\K_\A}=1$, we have $r_{\K_\A}(f,0)=1$. The example can be modified to obtain a vector valued holomorphic function with similar properties. 

However, for a general Banach operator ideal $\A$, there exist a holomorphic mapping $f \in \H(\ell_1,X)$ which is $\A$-compact at the origin with $r_{\K_\A}(f,0)=1$ but $f$ fails to be $\A$-compact at $e_1$, the first element of the canonical basis of $\ell_1$. To give this example with need some lemmas. 

\begin{lemma}\label{Lemma_A-null_to_0} Let $X$ be a Banach space, $\A$ a Banach operator ideal and $(x_j)_j \subset X$ an $\A$-null sequence. Then $\lim_{m\rightarrow \infty} \m_\A(\co\{(x_j)_{j>m}\},X)=0$.
\end{lemma}
\begin{proof}
Take $(x_j)_j \subset X$ an $\A$-null sequence. By \cite[Lemma 1.2]{CaSt} there exist a Banach space $Z$, an operator $T\in \A(Z;X)$ and a null sequence $(z_j)_j \subset Z$ such that $T(z_j)= x_j$ for all $j \in \Na$. Hence, given $\ep>0$ there exists $j_0 \in \Na$ such that $\|z_j\|\leq \frac{\ep}{\|T\|_\A}$.  Then we have
$$
\co\{(x_j)_{j>m}\}\subset \frac{\ep}{\|T\|_\A} T(\textstyle \frac{\|T\|_\A}{\ep} \co\{(z_j)_{z>m}\}),
$$ 
and since $\co\{(z_j)_{j>j_0}\}\subset \frac{\ep}{\|T\|_\A}B_Z$ is a relatively compact set,
$$
\m_\A(\co\{(x_j)_{j>m}\},X)\leq  \frac{\ep}{\|T\|_\A}\|T\|_{\A}=\ep.
$$
\end{proof}

\begin{lemma}\label{lemma_Sequence_m_A_infinite} Let $X$ be a Banach space, $\A$ and $\B$ Banach operator ideals and $K\subset X$ a relatively $\B$-compact set which is not relatively $\A$-compact. Then, there exist a $\B$-null sequence $(x_j)_j \subset X$ and a increasing sequence of integers $1=j_1<j_2<j_3 \ldots$ such that if
$$ 
L_m=\{x_{j_m},x_{j_m+1},\ldots,x_{j_{m+1}-1}\},
$$

\begin{enumerate}[\upshape (a)]
\item $\lim_{m\rightarrow \infty} \m_\B(L_m,X)=0$.
\item $\lim_{m\rightarrow \infty} \m_\A(L_m,X)=\infty$.
\end{enumerate}
In particular, the sequence $(x_j)_j$ is not $\A$-null.

\end{lemma}

\begin{proof}
Let $K\subset X$ be a relatively $\B$-compact set which is not relatively $\A$-compact and take a $\B$-null sequence $(x_j)_j$ such that $K\subset \co\{(x_j)_j)\}$. Let $(\gamma_j)_j \in B_{c_0}$ be a sequence such that, if $\tilde x_j=\frac{x_j}{\gamma_j}$ for all $j \in \Na$, then $(\tilde x_j)_j \in c_{0,\B}(X)$. Note that by the above lemma, for every increasing sequence of integers $1=j_1<j_2<j_3 \ldots$, the sequence
$$
L_m=\{x_{j_m},x_{j_m+1},\ldots,x_{j_{m+1}-1}\},
$$
satisfies $\lim_{m\rightarrow \infty} \m_\B(L_m,X)=0$. Then, the result follows once we show that there exists an increasing sequence $1=j_1<j_2<j_3 \ldots$ such that $\lim_{m\rightarrow \infty} \m_\A(L_m,X)=\infty$. Suppose that such sequence does not exist. Then, there is $C>0$ such that, for every choice $\tilde x_{j_1}, \tilde x_{j_2},\ldots, \tilde x_{j_j}$, 
$$
\m_\A(\{\tilde x_{j_1}, \tilde x_{j_2},\ldots, \tilde x_{j_j}\},X)\leq C.
$$
Consider the linear operators $R\colon \ell_1 \rightarrow \ell_1$ and $T\colon \ell_1 \rightarrow X$ which are defined on the elements of the canonical basis $(e_j)_j$ by
$$
R(e_j)=\gamma_j^{1/2} e_j \quad {\rm and} \quad T(e_j))\tilde x_j
$$
and extended by linearity and density. Since $T(B_{\ell_1})\subset \co\{(\tilde x_j)_j\}$, $T\in \K_\B(\ell_1;X)$. Also, $R\in \overF(\ell_1;\ell_1)$ and 
$$
K\subset T\circ R\circ R (B_{\ell_1}).
$$
Denote by $S_j\colon \ell_1\rightarrow X$ the linear operator defined as $S_j = T\circ \pi_j \circ R\circ R$, where $\pi_j\colon \ell_1\rightarrow \ell_1$ is the projection onto the first $j$ coordinates of $\ell_1$. Note that, for each $j \in \Na$, $S_j \in \F(\ell_1;X)$. Since
$$
\|S_{j_0}-T\circ R\circ R\|_{\K_\B}=\|T \circ (Id_{\ell_1}-\pi_{j_0})\circ R \circ R\|_{\K_\B}\leq \|T\|_{\K_\B} \sup_{j>j_0}|\gamma_j|,
$$
and $(\gamma_j)_j \in c_0$, the sequence $(S_j)_j$ converge to $T\circ R\circ R$. Hence, if we show that $(S_j)_j$ is a $\|\cdot\|_{\K_\A}$-Cauchy sequence, it would imply that $T\circ R\circ R \in \K_\A(\ell_1;X)$ and, as consequence, that $K$ is relatively $\A$-compact, which is a contradiction. Now, for $j<k$
$$
S_k-S_j=T\circ (\pi_k-\pi_j)\circ R \circ R=T\circ (\pi_k-\pi_j)\circ R \circ R \circ (\pi_k-\pi_j),
$$
therefore
\begin{equation}\label{eq1}
\|S_k-S_j\|_{\K_\A} \leq \|T\circ (\pi_k-\pi_j)\circ R\|_{\K_\A} \|R\circ (\pi_k-\pi_j)\|.
\end{equation}
Note that
$$
\begin{array}{rl}
\|T\circ (\pi_k-\pi_j)\circ R\|_{\K_\A}=&\m_\A(T\circ (\pi_k-\pi_j)\circ R(B_{\ell_1});X)\\
=&\m_\A(\gamma_j^{1/2} \tilde x_j, \gamma_{j+1}^{1/2} \tilde x_{j+1},\ldots, \gamma_{k-1}^{1/2}\tilde x_{k-1},X)\\
\leq&\m_\A(\tilde x_j, \tilde x_{j+1},\ldots, \tilde x_{k-1},X)\\
\leq & C.
\end{array}
$$
Then, from \eqref{eq1} it follows that
$$
\|S_k-S_j\|_{\K_\A}\leq C \|R\circ (\pi_k-\pi_j)\|=C \sup \{ |\gamma_j|^{1/2},|\gamma_{j+1}|^{1/2},\ldots |\gamma_{k-1}|^{1/2} \}.
$$
Since $(\gamma_j)_j \in B_{c_0}$, the last inequality shows that $(S_j)_j$ is a $\|\cdot\|_{\K_\A}$-Cauchy sequence, as we wanted to show.
\end{proof}

\begin{lemma}\label{Lemma_medida_finitos} Let $X$ be a Banach spaces, $\A$ a Banach operator ideal and $x_1,\ldots, x_m \in X$. Fix $n \in \Na$ and consider the set
$$
L=\{\sum_{j=1}^{m} \alpha_j^n x_j \colon (\alpha_j)_j \in B_{\ell_1}\}.
$$
Then $L$ is a relatively $\A$-compact set and $\m_\A(L,X)=\m_\A(\{x_1,\ldots,x_m\},X)$.
\end{lemma}
\begin{proof}
First note that since $L\subset E$ for some finite dimensional subspace $E$ of $X$ and $L$ is bounded, then it is relatively $\A$-compact. Now, if $(\alpha_j)_j \in \ell_1$, $(\alpha^m_j)_j \in \ell_1$ for any $m\geq 1$ and $\|(\alpha^m_j)_j\|_{\ell_1}\leq \|(\alpha_j)_j\|_{\ell_1}$. Thus 
$$
\{x_1,\ldots,x_m\}\subset L \subset \co\{x_1,\ldots,x_m\},
$$
and the result follow since $\m_\A(\{x_1,\ldots,x_m\},X)=\m_\A(\co\{x_1,\ldots,x_m\},X)$.
\end{proof}

Now we give an example of a holomorphic function which is $\A$-compact at the origin but this property does not extend beyond the ball with center at $0$ and radius $r_{\K_\A}(f,0)$. Since this function is $\A$-compact at $0$, in particular is compact at $0$ and, by \cite[Proposition 3.4]{AS}, the function is compact. The example shows that, with some hypothesis, the function can be taken to be $\B$-compact for some Banach operator ideal $\B$. The construction is based on \cite[Example~11]{DIN3} and \cite[Example~3.8]{LaTur1}. As usual $(e_j)_j$ is the canonical basis of $\ell_1$ and $(e'_j)_j$ is the sequence of coordinate functionals on $\ell_1$.

\begin{example}\label{exam_A-compact_at_0} Let $X$ be a Banach space, $\A$ and $\B$ Banach operator ideals such that there exists a relatively $\B$-compact set which not is $\A$-compact. Then, there exist a $\B$-compact holomorphic function $f \in \H_{\K_\B}(\ell_1;X)$ such that $f$ is $\A$-compact at $0$, but fails to be $\A$-compact at $e_1$.
\end{example}
\begin{proof}
By assumption, as a consequence of Lemma~\ref{lemma_Sequence_m_A_infinite}, there exist an $\B$-null sequence $(x_j)_j \subset X$ which is not $\A$-null and a increasing sequence $1=j_1<j_2,\ldots$ such that, with
$$
L_k=\{x_{j_k},x_{j_k+1},\ldots,x_{j_{k+1}-1}\},
$$
we have $\lim_{k\rightarrow \infty}(\m_\B(L_k,X)^{1/k}=0$ and $\lim_{k\rightarrow \infty} \m_\A(L_k,X)^{1/k}=\infty$.  For each $j \in \Na$ such that $j_k\leq j\leq j_{k+1}-1$ consider
$$
\tilde x_j=\frac{x_j}{\m_\A(L_k,X)} \quad {\rm and} \quad 
\tilde L_k=\{\tilde x_{j_k}, \tilde x_{j_k+1},\ldots, \tilde x_{j_{k+1}-1}\}.
$$
Routine arguments show that $\lim_{k\rightarrow \infty}(\m_\B(\tilde L_k,X)^{1/k}=0$ and $\lim_{k\rightarrow \infty} \m_\A(\tilde L_k,X)^{1/k}=1$. Fixed $n\geq 2$, define the $n$-homogeneous polynomial $P_n \in \p^n(\ell_1;X)$ by
$$
P_n(z)=e'_1(z)^{n-2} \sum_{j=j_n}^{j_{n+1}-1} e'_j(z)^2 \tilde x_j.
$$
Note that 
$$
P_n(B_{\ell_1})\subset \{\sum_{j=j_n}^{j_{n+1}-1} e'_j(z)^2 \tilde x_j\colon z \in B_{\ell_1}\},
$$
then $P_n \in \p^n_{\K_\B}(\ell_1,X)$ and, by Lemma~\ref{Lemma_medida_finitos}, $\|P_n\|_{\K_\B}\leq  \m_\B(\tilde L_n,X)$. Since $\|P_n\|\leq \|P_n\|_{\K_\B}$, then $\lim_{n\rightarrow \infty} \|P_n\|^{1/n}=0$ and following \cite[Example~5.4]{Mu}, we may define an entire function $f \in \H(\ell_1;X)$ as $f(z)=\sum_{n\geq 2} P_n(x)$. Moreover, since for every $n\geq 2$, $P_nf(0)=P_n$, $r_\B(f,0)=1/\m_\B(\tilde L_n,X)^{1/n}=\infty$. Then, by Corollary~\ref{Coro_A-comp1pt_A-comp}, $f \in \H_{\K_\B}(\ell_1;X)$. Also, note that $P_nf(0) \in \p^n_{\K_\A}(\ell_1;X)$ for all $n\geq 2$ and, again by Lemma~\ref{Lemma_medida_finitos} $\limsup \|P_nf(0)\|_{\K_\A}^{1/n}\leq \m_\A(\tilde L_n,X)^{1/n} =1$. Hence $r_{\K_\A}(f,0)\geq 1$ and $f$ is $\A$-compact at $0$. In order to see that $r_{\K_\A}(f,0)= 1$, fix $n\geq 2$ and $\ep>0$. Take $z_j \in B_{\ell_1}$ such that
$$
e'_1(z_j)=1-\ep, \quad e'_j(z_j)=\ep \quad {\rm and} \quad e'_k(z_j)=0 \quad {\rm for} j\neq k.
$$
Then, for each $j_n\leq j < j_{n+1}$ , $P_n(z_j)=(1-\ep)^{n-1} \ep^2 \tilde x_j$, which implies that
$$
(1-\ep)^{n-2}\ep^2\tilde L_n \subset P_n(B_{\ell_1}).
$$
Finally, note that 
$$
\m_\A((1-\ep)^{n-2}\ep^2\tilde L_n,X)= (1-\ep)^{n-2}\ep^2 \m_\A(\tilde L_n,X)\leq \m_\A(P_n(B_{\ell_1}),X)= \|P_n\|_{\K_\A}.
$$
Then $(1-\ep)\leq \limsup \|P_n\|_{\K_\A}^{1/n}$ and $r_{\K_\A}(f,0)\geq 1/(1-\ep)$ for all $\ep>0$, giving that $r_{\K_\A}(f,0)=1$.

In order to show that $f$ is not $\A$-compact at $e_1$, thanks to Proposition~\ref{Prop_A-compact_equiv} it enough to show that the $2$-homogeneous polynomial $P_2f(e_1)$ in not $\A$-compact. On the one hand, by \cite[Example~5.4]{Mu}, we have that
\begin{equation}\label{eq2}
P_2f(e_1)(z)=\sum_{n=2}^{\infty} \genfrac{(}{)}{0pt}{}{n}{2}\vP_n(e_1^{n-2},z^2).
\end{equation}
On the other hand, if we consider the $n$-linear operator $A_n \in \L^n(\ell_1;X)$ given by
$$
A_n(z_1,\ldots,z_n)=e'_1(z_1)\ldots e'_1(z_{n-2})\sum_{j=j_n}^{j_{n+1}-1} e'_j(z_{n-1})e'_j(z_n)\tilde x_j,
$$
then we have $P_n(z)=A_n(z,\ldots,z)$ and, denoting by $A^{\sigma}_n$ the $n$-linear operator given by
$$
A_n^{\sigma}(z_1,\ldots,z_n)=A_n(z_{\sigma(1)},\ldots,z_{\sigma(n)}), 
$$
where $\sigma \in \mathcal S_n$ is a permutation of $n$-elements, by~\cite[Proposition~1.6]{Mu} we have 
$$
\vP_n(e_1^{n-2},z^2)=\frac{1}{n!}\sum_{\sigma \in \mathcal S_n} A^{\sigma}_n (e_1^{n-2},z^2).
$$
Since $A_n(z_1,\ldots,z_{n-2},e_1,z_{n-1})=A_n(z_1,\ldots,z_{n-1},e_1)=0$ for all $z_1,\ldots,z_{n-1} \in \ell_1$ and $A_n(e_1^{n-2},z^2)=\sum_{j=j_n}^{j_{n+1}-1}e'_j(z)^2\tilde x_j$, we obtain that
\begin{equation}\label{eq3}
\vP_n(e_1^{n-2},z^2)=\frac{2 (n-2)!}{n!} \sum_{j=j_n}^{j_{n+1}-1} e'_j(z)^2\tilde x_j.
\end{equation}
Combining~\eqref{eq2} with \eqref{eq3}, we see that
$$
P_2f(e_1)(z)=\sum_{n\geq2}\sum_{j=j_n}^{j_{n+1}-1} e'_j(z)^2 \tilde x_j.
$$
Finally, since for every $j \in \Na$, $P_2f(e_1) e_j=\tilde x_j$, it follows that the sequence $(\tilde x_j)_j \subset P_2f(e_1)(B_{\ell_1})$ and, since $(\tilde x_j)_j $ is not $\A$-null, $P_2f(e_1)$ fails to be $\A$-compact, and the result is proved.
\end{proof}

The next proposition states that, for a Banach operator ideal $\A$, $\K_\A$ is a non-closed Banach operator ideal if and only if there exist a Banach space $X$ and a relatively compact set $K\subset X$ such that $K$ is not $\A$-compact. In particular, the above example, together with the next proposition, show that, we can not generalize (iv) implies (i) of Proposition~\ref{Prop_AS} in the case of non-closed $\A$-compact operators ideals. 

\begin{proposition}\label{Prop_Cerrados}
Let $\A$ be a Banach operator ideal. Then $\K_\A$ is a non-closed Banach operator ideal if and only if there exist a Banach space $X$ and a relatively compact set in $X$ which is not relatively $\A$-compact.
\end{proposition}
\begin{proof}
First, suppose that $\K_\A$ is non-closed. Then, there exist Banach spaces $X$ and $Y$, a sequence of $\A$-compact operators $(T_n)_n \in \K_\A(Y;X)$ and a operator $\in \L(Y;X)$ such that $(T_n)_n$ converge in the usual norm to $T$, but $T$ is not $\A$-compact. Since every $\A$-compact operator is, a compact operator, it follows that $T \in \K(Y;X)$. Hence, $K=T(B_Y)$ is a relatively compact set in $X$ which is not relatively $\A$-compact. For the converse, take $K\subset X$ a relatively compact set which is not $\A$-compact. By Lemma~\ref{lemma_Sequence_m_A_infinite}, there exist a sequence $(x_n)_n \in c_0(X)$ which is not $\A$-compact. Consider the operator $T\colon \ell_1 \rightarrow X$ given by $T(e_n)=x_n$ and extended by linearity and density. Is clear that $T$ is a compact operator which is not $\A$-compact. If $\pi_n\colon \ell_1 \rightarrow \ell_1$ is the projection to the first $n$-coordinates, then the sequence $(T\circ \pi_n)_n \subset \F(\ell_1;X)$ (hence it belongs to $\K_\A(\ell_1;X)$) and converges in the usual norm to $T$. This shows that $\K_\A$ is a non-closed operator ideal.
\end{proof}

Now, let us consider the following space. For Banach spaces $X$ and $Y$ and a Banach operator ideal $\A$, 
$$
\H^{\p}_{\K_\A}(X;Y)=\{f \in \H(X;Y) \colon P_nf(x) \in \p^n_{\K_\A}(X;Y) \ \forall \ x \in X, \ \forall \ n \in \Na\}.
$$
By Proposition~\ref{Prop_A-compact_equiv}, we have the inclusion $\H_{\K_\A} \subset \H^{\p}_{\K_\A}$ and, in the case of $\A=\K$, $\H_\K=\H^{\p}_{\K}$ \cite[Proposition~3.4]{AS}. However, under certain conditions, the identity $\H_{\K_\A}=\H^{\p}_{\K_\A}$ does not hold, as we show in the next example. The following construction is based on \cite[Example~3.7]{LaTur1} and \cite[Example~10]{DIN3}.

\begin{example}\label{exam_fun_no_A_compacta}
Let $X$ be a Banach space, $\A$ and $\B$ Banach operator ideals such that there exists a relatively $\B$-compact set which is not $\A$-compact. Then there exists an $\B$-compact holomorphic mapping $f\in \H_{\K_\B}(\ell_1;X)$ such that $f \in \H^{\p}_{\K_\A}(\ell_1;X)$, but $f$ is not $\A$ compact at any $z \in \ell_1$.
\end{example}
\begin{proof}
By Lemma~\ref{lemma_Sequence_m_A_infinite}, there exist a $\B$-null sequence $(x_j)_j \subset X$ and a increasing sequence $1=j_1<j_2<j_3\ldots$ such that, if $L_n=\{x_{j_n},x_{j_n+1},\ldots,x_{j_{n+1}-1}\}$, then
$$
\lim_{m\rightarrow \infty} \m_\B(L_m,X)^{1/m}=0 \quad {\rm and} \quad \lim_{m\rightarrow \infty} \m_\A(L_m,X)^{1/m}=\infty.
$$
For each $n\geq 1$, consider the $n$-homogeneous polynomial $P_n \in \p^n(\ell_1;X)$ defined by
$$
P_n(z)=\sum_{j=j_n}^{j_{n+1}-1}e'_j(z) x_j.
$$
Then
$$
\|P_n\|_{\K_\B}=\m_\B(\{\sum_{j=j_n}^{j_{n+1}-1}e'_j(z)^n x_j\colon z \in B_{\ell_1}\};X)=\m_\B(L_n;X), 
$$
where the last equality follows from Lemma~\ref{Lemma_medida_finitos}. Since $\|P_n\|\leq \|P_n\|_{\K_\B}$, we have
$$
\limsup\|P_n\|^{1/n} \leq \limsup\|P_n\|_{\K_\B}^{1/n}= \limsup \m_\B(L_n,X)^{1/n}=0.
$$
Now, combining \cite[Example~5.4]{Mu} and Corollary~\ref{Coro_A-comp1pt_A-comp}, we see that the holomorphic mapping $f=\sum_{n=1}^{\infty} P_n$ is well defined and is $\B$-compact. Also, note that for all $n\in \Na$, $P_n=P_nf(0)$ is $\A$-compact and, again by Lemma~\ref{Lemma_medida_finitos}, we have
$$
\limsup\|P_nf(0)\|_{\K_\A}^{1/n}= \limsup \m_\A(L_n,X)^{1/n}=\infty.
$$
Then, by Proposition~\ref{Prop_A-compact_equiv}, $f$ is not $\A$-compact at $0$.

Now, take any $z_0\in \ell_1$, $z_0\neq 0$ and fix $n_0 \in \Na$. Since the $n$-symmetric multilineal operator associated to $P_n$ is given by
$$
\vP_n(z_1,\ldots,z_n)=\sum_{j=j_n}^{j_{n+1}-1} e'_j(z_1)e'_j(z_2)\ldots e'_j(z_n) x_j,
$$
for all $z \in B_{\ell_1}$ we have

$$
\begin{array}{rl}
\d P_{n_0}f(z_0)(z)=&\d \sum_{j=n_0}^{\infty}\genfrac{(}{)}{0pt}{}{j}{n_0}\vP_j(z_0^{j-n_0},z^{n_0})\\
                =&\d \sum_{j=n_0}^{\infty}\genfrac{(}{)}{0pt}{}{j}{n_0} \sum_{j=j_n}^{j_{n+1}-1} e'_j(z_0)^{j-n_0}e'_j(z)^{n_0}x_j.               
\end{array}
$$
We claim that $P_{n_0}f(z_0)$ is $\A$-compact. In fact, 
\begin{equation}\label{eq4}
\begin{array}{rl}
\m_\A(P_{n_0}f(z_0)(B_{\ell_1},X)&\d \leq \sum_{j=n_0}^{\infty}\genfrac{(}{)}{0pt}{}{j}{n_0} \m_\A(\{\sum_{j=j_n}^{j_{n+1}-1} e'_j(z_0)^{j-n_0}e'_j(z)^{n_0}x_j\} \colon z \in B_{\ell_1}\},X)\\
&\d \leq \sum_{j=n_0}^{\infty}\genfrac{(}{)}{0pt}{}{j}{n_0} \sum_{j=j_n}^{j_{n+1}-1} |e'_j(z)|^{j-n_0}\|x_j\| \\
&\d \leq  C \sum_{j=n_0}^{\infty}\genfrac{(}{)}{0pt}{}{j}{n_0} \Big(\sum_{j=j_n}^{j_{n+1}-1} |e'_j(z_0)|\Big)^{j-n_0} 
\end{array}
\end{equation}
where $C=\sup_{j \in \Na} \|x_j\|$. Denote by $b_n=\sum_{j=j_n}^{j_{n+1}-1} |e'_j(z)|$. Since $z_0 \in \ell_1$, $(b_n)_n \in \ell_1$. This fact, shows that we can apply D'Alambert's criterion to the last series of inequality~\eqref{eq4} to show that it converge. Now, an application of Lemma~\ref{Lemma_Suma_A-compactos} shows that $P_{n_0}f(z_0)$ is $\A$-compact.

Finally, to see that $f$ is not $\A$-compact at $z_0$, just note that if we choose $j \in \Na$, $j_{n_0}\leq j<j_{n_0+1}$, then $P_{n_0}f(z_0)(e_j)=x_j$. Hence, $L_{n_0}\subset P_{n_0}f(z_0)(B_{\ell_1})$ and arguing as we did at the beginning of the example, we conclude that $f$ is not $\A$-compact at $z_0$.
\end{proof}

In particular, the above example, together with Proposition~\ref{Prop_Cerrados}, show that we can not generalize (ii) implies (i) and (iv) implies (iii) to the case of non-closed $\A$-compact operator ideals. Also, Example~\ref{exam_fun_no_A_compacta} shows that, in general, $\H_{\K_\A} \subsetneq \H^{\p}_{\K_\A}$. However, for $X$ and $Y$ Banach spaces, $\H_{\K_\A}(X;Y)$ is $\tau_{\omega}$-dense in $\H^{\p}_{\K_\A}(X;Y)$, where $\tau_{\omega}$ stands for the Nachbin topology. Following \cite[Proposition~3.47]{DIN3}, $\tau_{\omega}$ is a locally convex topology on $\H(X;Y)$ which is generated by the seminorms of the form 
$$
q(f)=\sum_{n=0}^{\infty} \sup_{x \in K+\alpha_n B_X} \|P_nf(0)(x)\|,
$$
where $K$ ranges over all absolutely convex compact subsets of $X$ and $(\alpha_n)_n$ over $c_0$.

\begin{proposition}\label{prop_Denso}
Let $X$ and $Y$ be Banach spaces and $\A$ a Banach operator ideal. Then, $\H_{\K_\A}(X;Y)$ is $\tau_{\omega}$-dense in  $\H^{\p}_{\K_\A}(X;Y)$.
\end{proposition}
\begin{proof}

Let $f_0 \in \H^{\p}_{\K_\A}(X;Y)$ and consider the $\tau_{\omega}$-continuous seminorm $q$ determined by $K$ and $(\alpha_n)_n$
$$
q(f)=\sum_{n=0}^{\infty} \sup_{x \in K+\alpha_n B_X} \|P_nf(0)(x)\|.
$$
Given $\ep>0$, take $n_0 \in \Na$ such that $\sum_{n=n_0}^{\infty} \sup_{x \in K+\alpha_n B_X} \|P_nf(0)(x)\|\leq \ep$ and consider $\tilde f_0(x)=\sum_{n=0}^{n_0-1} P_nf(0)(x)$. Note that $\tilde f_0 \in \H_{\K_\A}(X;Y)$ since it is a finite sum of $\A$-compact polynomials. Then, the result follows from the inequality
$$
q(f_0 - \tilde f_0)=\sum_{n=n_0}^{\infty} \sup_{x \in K+\alpha_n B_X} \|P_nf(0)(x)\|\leq \ep.
$$
\end{proof}
\subsection*{Acknowledgements} I am indebted to my advisor Silvia Lassalle for her valuable and unselfish comments and suggestions, which substantially improved this article . I would also like to thank Richard Aron and Pilar Rueda for providing me with their article \cite{AR3}.


\begin{thebibliography}{99}

\bibitem{AB} Aron, R. and Berner, P.: A Hahn-Banach extension theorem for analytic mappings, Bull. Soc. Math. France, 106  (1), (1978), 3--24.

\bibitem{AMR} Aron, R., Maestre, M. and Rueda, P.: $p$-compact holomorphic mappings, RACSAM, 104  (2), (2010), 353--364.

\bibitem{AR1} Aron, R. and Rueda, P.: $p$-compact homogeneous polynomials from an ideal point of           view,Function spaces in modern analysis, 61--71, Contemp. Math., 547, Amer. Math. Soc., Providence, RI  (2011).

\bibitem{AR2} Aron, R. and Rueda, P.: Ideals of homogeneous polynomials, Publ. Res. Inst. Math. Sci., 48 (4), (2012), 957--969.

\bibitem{AR3} Aron, R. and Rueda, P.:$I$- Bounded holomorphic functions, J. Nonlinear Convex A., In press.

\bibitem{AS} Aron, R. and Schottenloher, M.: Compact holomorphic mappings on Banach spaces and the        approximation property, J. Functional Analysis, 21 (1), (1976), 7--30.

\bibitem{BoBrJuPe} Botelho, G.,  Braunss, H. A., Junek, H. and Pellegrino, D.: Holomorphy types and ideals of multilinear mappings, Studia Math., 177 (1), (2006), 43--65.

\bibitem{BPR} Botelho, G., Pellegrino, D. and Rueda, P.: On composition ideals of multilinear mappings and homogeneous polynomials, Publ. Res. Inst. Math. Sci., 43 (4), (2007), 1139--1155.

\bibitem{CaDiMu} Carando, D., Dimant, V. and Muro, S.: Coherent sequences of polynomial ideals on Banach spaces, Math. Nachr., 282  (8), (2009), 1111-1133.

\bibitem{CaSt} Carl, B. and Stephani, I.: $\mathcal A$-compact operators, generalized entropy numbers and
entropy ideals, Math. Nachr., 119, (1984), 77--95.


\bibitem{DG} Davie, A. and Gamelin, T.: A theorem on polynomial-star approximation, Proc. Amer. Math. Soc., 106 (2), (1989), 351--356.

\bibitem{DF} Defant, A. and Floret, K.: Tensor norms and operator ideals,  North-Holland Mathematics Studies 176, North-Holland Publishing Co., Amsterdam (1993).
  
\bibitem{djt} Diestel, J., Jarchow, H. and Tonge, A.: Absolutely summing operators, Cambridge Studies in Advanced Mathematics, 43, Cambridge University Press, Cambridge (1995).

\bibitem{DIN} Dineen, S.: Complex analysis on infinite-dimensional spaces, Springer Monographs in Mathematics, Springer-Verlag London, Ltd., London (1999).

\bibitem{DIN3} Dineen, S.: Holomorphic functions on a Banach space, Bull. Amer. Math. Soc., 76, (1970), 883--886.

\bibitem{Flo97} Floret, K.: Natural norms on symmetric tensor products of normed spaces, Proceedings of the Second International Workshop on Functional Analysis (Trier, 1997), Note Mat., 17, (1997), 153--188.

\bibitem{GonGu00} Gonz{\'a}lez, M. and Guti{\'e}rrez, J.: Surjective factorization of holomorphic mappings, Comment. Math. Univ. Carolin., 41 (3),  (2000), 469--476.

\bibitem{LaTur1} Lassalle, S. and Turco, P.: On $p$-compact mappings and the $p$-approximation              property, J. Math. Anal. Appl., 389 (2), (2012), 1204--1221.

\bibitem{LaTur2} Lassalle, S. and Turco, P.: The Banach ideal of $\mathcal A$-compact operators and related approximation properties, J. Funct. Anal., 265 (10), (2013), 2452--2464.

\bibitem{Mu} Mujica, J.: Complex analysis in Banach spaces, North-Holland Mathematics Studies, 120, North-Holland Publishing Co., Amsterdam (1986).

\bibitem{Pie} Pietsch, A.: Operator ideals, North-Holland Mathematical Library 20, North-Holland Publishing Co., Amsterdam-New York (1980).

\bibitem{RYAN} Ryan, R.: Introduction to tensor products of Banach spaces, Springer Monographs in Mathematics, Springer-Verlag London, Ltd., London (2002).

\bibitem{Ryan_Wc} Ryan, R.: Weakly compact holomorphic mappings on Banach spaces, Pacific J. Math., 131 (1), (1988) 179--190.

\bibitem{SiKa} Sinha, D. and Karn, A.: Compact operators whose adjoints factor through subspaces of $l_p$, Studia Math., 150 (1), (2002), 17--33.

\bibitem{Ste} Stephani, I. Generating systems of sets and quotients of surjective operator ideals, Math. Nachr., 99, (1980), 13--27.
\end{thebibliography}
\end{document}